\documentclass[12pt]{amsart}
 \usepackage{amsfonts,amssymb,pdfsync}

\usepackage{mathrsfs}
\let\mathcal\mathscr

\usepackage[all,ps,cmtip]{xy} 

\def\Z{{\bf Z}}

\def\C{{\bf C}}

\def\P{{\bf P}}
\def\T{{\bf T}}

\def\ppavs{principally polarized abelian varieties}

\def\phi{{\varphi}}

\def\cI{\mathcal{I}}

\def\cA{\mathcal{A}}
\def\cB{\mathcal{B}}

\def\cO{\mathcal{O}}

\def\cP{\mathcal{P}}

\def\cC{\mathcal{C}}
\def\cM{\mathcal{M}}
\def\cN{\mathcal{N}}

\def\cX{\mathcal{X}}

\def\lra{\longrightarrow}
\def\llra{\hbox to 10mm{\rightarrowfill}}
\def\lllra{\hbox to 15mm{\rightarrowfill}}

\def\llla{\hbox to 10mm{\leftarrowfill}}
\def\lllla{\hbox to 15mm{\leftarrowfill}}

\def\dra{\dashrightarrow}
\def\isom{\simeq}
\def\eps{\varepsilon}
\def\ie{\hbox{i.e.}}

\def\vide{\varnothing}

\DeclareMathOperator{\isomto}{\stackrel{{}_{\scriptstyle\sim}}{\to}}

\DeclareMathOperator{\isomlra}{\stackrel{{}_{\scriptstyle\sim}}{\lra}}
\DeclareMathOperator{\rank}{rank}
\DeclareMathOperator{\Ker}{ker}

\DeclareMathOperator{\Pic}{Pic}

\DeclareMathOperator{\Nm}{Nm}

\DeclareMathOperator{\Sing}{Sing}

\DeclareMathOperator{\Prym}{Prym}

\DeclareMathOperator{\Aut}{Aut}
\DeclareMathOperator{\PGL}{PGL}

\def\llra{\hbox to 10mm{\rightarrowfill}}
\def\lllra{\hbox to 15mm{\rightarrowfill}}

\newtheorem{lemm}{Lemma}[section]
\newtheorem{theo}[lemm]{Theorem}
\newtheorem{coro}[lemm]{Corollary}
\newtheorem{prop}[lemm]{Proposition}

\theoremstyle{definition}

\newtheorem{rema}[lemm]{Remark}
\newtheorem{conj}[lemm]{Conjecture}

\theoremstyle{remark}
\newtheorem*{remark*}{Remark}
\newtheorem*{note*}{Note}

\begin{document}
\title{On nodal prime Fano threefolds of degree 10}
\author[O. Debarre]{Olivier Debarre}
\thanks{O. Debarre and L. Manivel are part of the  project VSHMOD-2009   ANR-09-
BLAN-0104-01.}
\address{D\'epartement de  Math\'ematiques et Applications -- CNRS UMR 8553,
\'Ecole Normale Sup\'erieure,
45 rue d'Ulm, 75230 Paris cedex 05, France}
\email{{\tt Olivier.Debarre@ens.fr}}
\author[A. Iliev]{Atanas Iliev}
\address{Institute of Mathematics,
Bulgarian Academy of Sciences,
Acad. G. Bonchev Str., bl. 8, 
1113 Sofia, Bulgaria}
\email{{\tt ailiev@math.bas.bg}}
\author[L. Manivel]{Laurent Manivel}
\address{Institut Fourier,  
Universit\'e de Grenoble I et CNRS,
BP 74, 38402 Saint-Martin d'H\`eres, France}
\email{{\tt Laurent.Manivel@ujf-grenoble.fr}}

\def\moins{\mathop{\hbox{\vrule height 3pt depth -2pt
width 5pt}\,}}

\begin{abstract}
We study the geometry and the period map of nodal     complex prime Fano threefolds  with index 1 and degree 10. We show that these threefolds are birationally isomorphic to Verra solids, \ie, hypersurfaces of bidegree $(2,2)$ in $ \P^2\times \P^2$. Using Verra's results on the period map for these solids and on the Prym map for double \'etale covers of plane sextic curves, we prove that the fiber of the   period map for our nodal threefolds is birationally the union of two surfaces, for which we give several descriptions. 
This result is the analog in the nodal case of a result of \cite{dim} in the smooth case.
\end{abstract}

\maketitle
\hfill{\it Dedicated to Fabrizio Catanese on his 60th birthday}


\tableofcontents
 
 \section{Introduction}\label{intro}
 
\noindent{\bf Nodal prime Fano threefolds of degree 10.} There are 10 irreducible families of   
smooth complex Fano threefolds $X$ with Picard   group $ \Z[K_X]$, one for each degree $(-K_X)^3=2g-2$, where $g\in\{2,3,\dots,10,12\}$. 
  The article is a sequel to \cite{dim}, where we studied the geometry and the period map of {\em smooth}      complex Fano threefolds $X$ with Picard group $ \Z[K_X]$ and degree $ 10$. Following again Logachev (\cite{lo}, \S5), we study here 
  complex Fano threefolds $X$ with   Picard group $ \Z[K_X]$ and  degree $10$ which are general {\em  with one node $O$.} They are degenerations of their smooth counterparts and their geometry is made easier to study by the fact that they are (in two ways) birationally conic bundles over $\P^2$.

  \smallskip
\noindent{\bf Two conic bundle structures.} More precisely, the nodal variety $X$ is anticanonically embedded in $\P^7$, and  it has long been known  that the projection from its node $O$ maps $X$ birationally onto a (singular) intersection of three quadrics   $X_O\subset \P^6$. The variety  $X_O$, hence also $X$,   is  therefore   (birationally) a conic bundle, with discriminant a septic curve $\Gamma_7$,  union of a line $\Gamma_1$ and a smooth sextic $\Gamma_6$ (see \cite{bbb}, 5.6.2), and associated connected double \'etale cover $\pi:\widetilde \Gamma_6\cup \Gamma_1^1\cup\Gamma_1^2\to \Gamma_6\cup \Gamma_1$ (\S\ref{g6}). On the other hand,  
 the
  ``double projection'' of $X$ from  $O$ (\ie, the linear projection from the 4-dimensional embedded Zariski tangent space at $O$) is also (birationally) a conic bundle, with discriminant curve another smooth  plane sextic $\Gamma_6^\star$ and associated connected double \'etale cover $\pi^\star:\widetilde \Gamma_6^\star\to \Gamma_6^\star$ (\S\ref{pw}).  
 
  \smallskip
\noindent{\bf Birational isomorphism with Verra solids.}  We show (Theorem \ref{nodver}) that these two conic bundle structures define a birational map from $X$ onto a (general)   hypersurface $T\subset \P^2\times \P^2$ of bidegree $(2,2)$. These threefolds $T$ were studied by Verra in \cite{ver}; both projections to $\P^2$ are conic bundles and define connected  double \'etale covers $\pi$ and $\pi^\star$ of the discriminant curves, which are nonisomorphic smooth sextics. In particular, the associated Prym varieties $\Prym(\pi)$ and $\Prym(\pi^\star)$ are isomorphic (to the intermediate Jacobian   $J(T)$) and Verra   showed that the Prym map from the space of connected double \'etale covers of   smooth plane sextics to the moduli space of 9-dimensional \ppavs\ has degree   2.
  
  \smallskip
\noindent{\bf Fibers of the period map.}  This information is very useful for the determination of the fiber of the period map of our nodal threefolds $X$, \ie, for the description of all nodal threefolds of the same type with     intermediate Jacobian isomorphic to $J(X)$. This $J(X)$ is a 10-dimensional  extension by $\C^*$ of the intermediate Jacobian of its normalization, which is also the intermediate Jacobian of $T$, and the extension class depends only on $T$ (\S\ref{epm}). A nodal $X$ can be uniquely reconstructed from the data of a general sextic $\Gamma_6$ and an even   thetacharacteristic $M$ on the union of $\Gamma_6$ and a line $ \Gamma_1$: by work of Dixon, Catanese, and Beauville, the  sheaf $M$ (on $\P^2$) has a free resolution by a $7\times 7$ symmetric matrix of linear forms, which defines the net of quadrics whose intersection is  $X_O$ (Theorem \ref{tcM} and Remark \ref{conv}). We show that given the double cover $\pi:\widetilde \Gamma_6 \to \Gamma_6$, the set of even thetacharacteristics $M$
on    $\Gamma_6\cup \Gamma_1$, which induce $\pi$ on $\Gamma_6$, is isomorphic to the quotient of the  {\em special surface} $S^{\rm odd}(\pi)$  (as defined in \cite{spe}) by its natural involution $\sigma$ (Proposition \ref{MS}). Together with Verra's result, this implies that {\em the general fiber of the period map for our nodal threefolds is birationally the union of the two surfaces} $S^{\rm odd}(\pi)/\sigma$ and $S^{\rm odd}(\pi^\star)/\sigma^\star$ (\S\ref{epm}).

 A result of Logachev (\cite{lo}, Proposition 5.8) says that the surface $S^{\rm odd}$ is also isomorphic to the minimal model $F_m(X)$ of the normalization of the Fano surface $F_g(X)$ of conics contained in $X$. On the one hand, we obtain the analog in the nodal case of the reconstruction theorem of \cite{dim}, Theorem 9.1: a general  nodal $X$ can be reconstructed from the surface $F_g(X)$. On the other hand, the present description of the fiber of the period map at a nodal $X$  fits in with the construction in \cite{dim} of two (proper smooth) surfaces in the fiber of the period map at a smooth $X'$, one of them isomorphic  to $F_m(X')/\iota$ (\cite{dim}, Theorem 6.4).
In both the smooth and nodal cases, the threefolds in the fiber of the period map are obtained one from another by explicit birational transformations called line and conic transformations (see \S\ref{lt} and \S\ref{ct}).
  
 Unfortunately, because of properness issues, we cannot deduce from our present results that a general fiber of the period map for smooth   prime Fano threefolds of degree 10 consists of just these two surfaces (although we prove here that these surfaces are distinct, which was a point missing from \cite{dim}), although we certainly think that this is the case.
 
  \smallskip
\noindent{\bf Singularities of the theta divisor.} The singular locus of the theta divisor of the intermediate Jacobian of a Verra solid was described in \cite{ver}. This description fits in with a conjectural description of the singular locus of the theta divisor of the intermediate Jacobian of a smooth prime Fano threefold of degree 10 that we give in \S\ref{sings}. This conjecture would imply   the conjecture about the general fiber of the period map mentioned above.

 \section{Notation}
 
  $\bullet$ As a general rule, $V_m$ denotes an $m$-dimensional (complex) vector space, $\P_m$ an $m$-dimensional  projective space, and $\Gamma_m$ a plane curve of degree $m$. We fix a $5$-dimensional complex vector space
  $V_5$  and we consider the Pl\"ucker   embedding $G(2,V_5) \subset \P_9=\P(\wedge^2 V_5)$ and its 
  smooth intersection $W$ with a general $\P_7$ (\S\ref{fw}).
  
  $\bullet$  $O\in W$ is a general point,   $\Omega\subset \P_9$ is a general quadric with vertex $O$, and $X=W\cap\Omega\subset \P_7$ is an anticanonically embedded    prime Fano threefold  with one node at $O$;  $\widetilde X\to X$ is the blow-up of $O$.

  $\bullet$   $p_O:\P_7\dra\P^6_O$ is the projection from $O$. We write  $ \Omega_O=p_O(\Omega)\subset \P^6_O $, general quadric, $ W_O=p_O(W)\subset \P^6_O $, base-locus of a pencil $\Gamma_1 $ of quadrics or rank 6, and $ X_O=p_O(X)\subset \P^6_O $, base-locus of the net $\Pi=\langle \Gamma_1,\Omega_O\rangle$ of quadrics.

  $\bullet$ $W_O$ contains  the $3$-plane $\P^3_W= p_O(\T_{W,O} )$ and $X_O$ contains the smooth quadric surface $Q=\P^3_W\cap \Omega_O$. The singular locus of $W_O$   is a rational normal cubic curve $C_O\subset \P^3_W$. The singular locus of $X_O$ is  $Q\cap C_O=\{s_1,\dots,s_6\}$.
  
    $\bullet$   $\widetilde \P^6_O\to  \P^6_O $ is the blow-up along $\P^3_W$, $\widetilde W_O\subset \widetilde \P^6_O$ is the (smooth) strict transform of $W_O$,
 and $\widetilde X_O\subset \widetilde W_O$   the (smooth) strict transform of $X_O$.

  $\bullet$ The projection
 $p_W:X\dra \P^2_W$ 
from the 4-plane $\T_{W,O}$  is a birational conic bundle with associated double \'etale cover   
   $\pi^\star:\widetilde \Gamma^\star_6\to\Gamma^\star_6$ and involution $\sigma^\star$.
   
   $\bullet$ $\Gamma_7=\Gamma_6\cup\Gamma_1\subset \Pi$ is the discriminant curve of the net of quadrics $\Pi$ that contain $X_O$, with associated double \'etale cover   
$\pi:\widetilde \Gamma_6\cup\Gamma_1^1\cup\Gamma_1^2\to \Gamma_6\cup\Gamma_1$  and involution $\sigma$.   We write $\{  p_1,\dots,  p_6\}=\Gamma_1 \cap  \Gamma_6$ 
 and $\{\tilde p_1,\dots,\tilde p_6\}=\Gamma_1^1\cap\widetilde \Gamma_6$.

 $\bullet$ $F_g(X)$ is the connected surface that parametrizes conics on $X$, with singular locus  the smooth connected curve $\widetilde\Gamma^\star_6$ of conics on $X$ passing through $O$,
 and
 $\nu:\widetilde F_g(X)\to   F_g(X)$ is  the normalization (\S\ref{cto}). The smooth surface $\widetilde F_g(X)$ carries an involution $\iota$; its minimal model is $\widetilde F_m(X)$.
 
 $\bullet$ $S^{\rm even}$ and $S^{\rm odd}$ are the special surfaces associated with $\pi$, with involution $\sigma$ (\S\ref{seso}). There is an isomorphism $  \rho:\widetilde F_m(X)\isomto   S^{\rm odd}$ (\S\ref{s117}).
 
 $\bullet$ $T\subset \P^2\times \P^2$ is a {\em Verra solid,} \ie, a smooth hypersurface of bidegree $(2,2)$.

 \section{The fourfolds $W$ and $W_O$}\label{sec3}\label{pow}
 
  \subsection{The fourfold $W$}\label{fw}
 As explained in \cite{dim}, \S{3}, the intersections
$$W=G(2,V_5)\cap \P_7 \subset \P(\wedge^2 V_5),
$$
whenever smooth and 4-dimensional, are all isomorphic under the action of $\PGL(V_5)$. They correspond dually to pencils $\P_7^\bot$ of skew-symmetric forms on $V_5$ which are all of maximal rank. The map that sends  a   form in the pencil to its kernel has image a smooth conic $c_U\subset    \P ( V_5 )$ that spans a $2$-plane $\P (U_3)$, where   $U_3\subset  V_5$ is the unique common maximal isotropic subspace to all forms in  the pencil   (see \S\ref{coow} for explicit equations).
  
  \subsection{Quadrics containing $W$}\label{iden} To any one-dimensional subspace $V_1\subset V_5$, one associates a Pl\"ucker quadric in $  |\cI_{G(2,V_5)}(2)|$ obtained as the pull-back of $G(2,V_5/V_1)$ by the rational map $\P(\wedge^2V_5)\dra \P(\wedge^2(V_5/V_1))$. This gives a linear map 
$$\gamma_G:\P(V_5) \lra |\cI_{G(2,V_5)}(2)|\isom \P^{49}$$
whose image consists of quadrics   of rank $6$. Since no such quadric contains $\P_7$ and $|\cI_W(2)|$ has dimension 4, we   obtain an isomorphism
\begin{equation}\label{iso}
\gamma_W:\P(V_5) \isomlra |\cI_W(2)|.
\end{equation}
The quadric   $\gamma_W([V_1])\subset \P_7$   has rank $6$ except if the vertex of $\gamma_G([V_1])$, which is the $3$-plane $\P(V_1\wedge V_5)$  contained in $G(2,V_5)$, meets $\P_7$ along a $2$-plane, which must be  contained in $W$. This happens if and only if $[V_1]\in c_U$ (\cite{dim}, \S 3.3).

Given a point $O\in W$, corresponding to a $2$-dimensional subspace $V_O\subset V_5$ , the quadrics that contain $W$ and are singular at $O$ correspond, via the isomorphism (\ref{iso}), to the projective line $\P(V_O)\subset \P(V_5)$.

\subsection{The $\P^2$-bundle  $\P(\cM_O)\to \P^2_O$}\label{pmo}

As in \cite{dim}, \S3.4, we define, for any hyperplane $V_4\subset V_5$, the 4-dimensional vector space $M_{V_4}=\wedge^2V_4\cap V_8$ (where $V_8\subset \wedge^2V_5$ is the vector space such that $\P_7=\P(V_8)$) and the quadric surface
  $$Q_{W,V_4}=G(2,V_4)\cap \P(M_{V_4})\subset W.
$$
Let $\P^2_O\subset\P(V_5^\vee)$ be the set  of hyperplanes $V_4\subset V_5$ that contain  $V_O$, and consider the rank-$3$ vector bundle $\cM_O\to \P^2_O$ whose fiber over $[V_4]$ is $M_{V_4}/\!\!\wedge^2\! V_O$ and the associated $\P^2$-bundle $\P(\cM_O)\to \P^2_O$.

  \subsection{The fourfold  $W_O$}\label{woo}
  Let $O$ be a point of $W$, let $V_O\subset V_5$ be the corresponding $2$-dimensional subspace,  and  let 
$p_O:\P_7\dra\P^6_O$ be the projection from $O$. We set
   $$W_O=p_O(W)\subset \P^6_O.$$
The group $\Aut(W)$ acts on $W$ with four orbits $O_1,\dots,O_4 $ indexed by their dimensions (\S\ref{coow}), so there are four different $W_O$. We will restrict ourselves to the (general) case $O\in O_4$, although similar studies can be made for the other orbits.

Since $O\in O_4$, the line $\P(V_O)\subset\P(V_5)$ does not meet the conic $c_U$ (\S\ref{coow}), hence all quadrics in the pencil  $\gamma_W(\P(V_O))$   are singular at $O$ and have rank $6$  (\S\ref{iden}). This pencil projects to a pencil of rank-6 quadrics in $\P^6_O$ whose base-locus contains  the fourfold $W_O$. Since $W_O$ has degree $4$, it is equal to this base-locus  and contains the $3$-plane $\P^3_W= p_O(\T_{W,O} )$. The locus of the  vertices of these quadrics is a rational normal cubic curve $C_O\subset \P^3_W$, which is the singular locus of $W_O$ and parametrizes lines in $W$ through $O$ (see \S\ref{coowo} for explicit computations).
 
 In fact, all pencils of rank-$6$ quadrics in $\P^6$ are isomorphic (\cite{hp}, Chapter XIII, \S11).
 In particular, all quadrics in the pencil have  a common
$3$-plane, and   the fourfold that they define in $\P^6$ is isomorphic to $W_O$.

\subsection{The $\P^2$-bundle $\widetilde W_O\to \P^2_W$}\label{p2w}

 Let $\P^2_W$   parametrize 5-planes in $\P_7$ that contain $\T_{W,O}$ (or, equivalently,  4-planes in $\P^6_O$ that contain $\P^3_W$). Let $\eps: \widetilde \P^6_O\to \P^6_O$ be the blow-up of $\P^3_W$, with $\widetilde \P^6_O\subset \P^6_O\times\P^2_W$, and let $\widetilde W_O\subset  \widetilde \P^6_O$ be the strict transform of $W_O$. 
 
 We will   prove in \S\ref{coowo}   that the projection  $\widetilde W_O\to \P^2_W$ 
  is a $\P^2$-bundle, and that
 $\widetilde W_O$ is smooth.  Furthermore, there is an isomorphism $\P^2_O\isom \P^2_W  $ such that  the induced  isomorphism $\P^2_O\times\P^6_O\isom \P^2_W  \times\P^6_O$, gives by restriction {\em an isomorphism 
between the $\P^2$-bundles $\P(\cM_O)\to \P^2_O$  and  $\widetilde W_O\to \P^2_W$} (see \S\ref{wtw}).  

Finally, the strict transform $\widetilde\P^3_W$ of $\P^3_W$ in $\widetilde W_O$ is 
the intersection of the exceptional divisor of $\eps$ with  $\widetilde W_O$; it has therefore dimension 3 everywhere. Since it contains the inverse image of the cubic curve $C_O\subset \P^3_W$, which is a surface (\S\ref{coowo}), it must be the blow-up of $ \P^3_W$ along $C_O $. The fibers of the projection 
$\widetilde\P^3_W\to \P^2_W$ are the bisecant lines to $C_O$.

\section{The threefolds $X$ and  $X_O$} \label{pox}

We consider here singular threefolds
$$X=G(2,V_5)\cap \P_7 \cap \Omega,
$$
where $\Omega$ is a quadric in $\P^9$, such that  {\em   $X$ has a unique singular point $O$, which is a node.} 

 \begin{lemm}[Logachev]
 The intersection $G(2,V)\cap \P_7 $ is smooth, hence isomorphic to $W$, and one may choose $\Omega$ to be a cone with vertex $O$.
 \end{lemm}
 
\begin{proof}
 We follow \cite{lo}, Lemmas 3.5 and 5.7.   If the intersection $W' =G(2,V)\cap \P_7 $ is singular, the corresponding pencil $\P_7^\bot$ of skew-symmetric forms on $V_5$ contains a  form of rank $\le 2$, and one checks that the singular locus of $W'$ is
$$\Sing (W')=\bigcup_{\omega\in \P_7^\bot,\ \rank(\omega)\le 2} G(2,\Ker(\omega))\cap \P_7,
$$
hence is a union of linear spaces of dimension $\ge 1$. In particular, the intersection  with the quadric $\Omega$ either has at least two singular points or a singular point which is not a node. It follows that $W=G(2,V)\cap \P_7 $ is smooth.

Consider now the map
\begin{eqnarray*}
|\cI_W(2)|&\dra&\left\{ \begin{matrix}\hbox{Hyperplanes in $T_{\P_7,O}$}\\\hbox{containing $T_{W,O}$}\end{matrix}\right\}\isom \P^2\\
\Omega'&\longmapsto&T_{\Omega',O}.
\end{eqnarray*}
It is not defined exactly when $\Omega'$ is singular at $O$, which happens along the projective line $\gamma_W(\P(V_O))$ (\S\ref{iden}). Since it is nonconstant, it is therefore surjective. 

Assume that $\Omega$ is smooth at $O$. Since $W\cap \Omega$ is singular at $O$, we have $T_{W,O}\subset T_{\Omega_O}$, hence there exists a quadric $\Omega'\supset W$ smooth at $O$ such that 
$T_{\Omega',O} =T_{\Omega,O}$. Some linear combination of $\Omega$ and $\Omega'$ is then singular at $O$ and still cuts out $X$ on $W$.   \end{proof}

 Conversely, we will from now on consider   a {\em general} quadric $\Omega$ with vertex $O $ in the orbit $O_4$. The intersection   $X=W\cap \Omega$ is then  smooth except for one node at $O$ and the group  $\Pic(X)$ is   generated by the class of $\cO_X(1)$.\footnote{The lines from the two   rulings of the exceptional divisor of the blow-up of $O$ in $X$ are       numerically, but  not algebraically, equivalent (see \S\ref{extex}; the reader will check that the proof of this fact does not use the fact that $X$ is locally factorial!), hence  the local ring $\cO_{X,O}$ is factorial (\cite{mor}, (3.31)) and the Lefschetz Theorem still applies (\cite{gro}, Exp.~XII, cor.~3.6; the hypotheses $H^1(X,\cO_X(-k))=H^2(X,\cO_X(-k))=0$ for all $k>0$ follow  from Kodaira vanishing on $W$).}

 We keep the notation of \S\ref{woo} and set $\Omega_O=p_O(\Omega)$ and $X_O=p_O(X)$.  Let    $ \widetilde X\to X$ be the blow-up of $O$.
 The projection $p_O$ from $O$ induces a morphism
$$\widetilde X\stackrel{\phi}{\twoheadrightarrow} X_O\subset W_O\subset \P^6_O$$
which is an isomorphism except on the union of the lines in $X$ through $O$. There are six such lines,  corresponding  to the six points $s_1,\dots,s_6$ of $\Sing(W_O)\cap \Omega_O$, which are the six singular points of $X_O$. 

 The threefold  $X_O \subset \P^6_O$ is the complete intersection of three quadrics and conversely, given the intersection of $ W_O$ with a general smooth quadric $\Omega_O$, its inverse image under the birational map $W\dra W_O$ is a variety of type  $\cX_{10}$ with a node at $O$.

 \subsection{The conic bundle   $p_W:X\dra \P^2_W$ and the double cover $\pi^\star:\widetilde \Gamma^\star_6\to\Gamma^\star_6$}\label{pw}
Keeping the notation of \S\ref{p2w},   consider the
projection
$$p_W:X\dra \P^2_W$$
from the 4-plane $\T_{W,O}$. It is also induced on  $X_O\subset  \P^6_O $ by the projection from the  $3$-plane $\P^3_W= p_O(\T_{W,O} )$, hence is a well-defined morphism on the strict transform $\widetilde X_O$ of $X_O$ in the blow-up $\widetilde \P^6_O$ of $ \P^6_O $ along $\P^3_W$, where $\widetilde \P^6_O\subset   \P^6_O\times \P^2_W$.

\begin{prop}\label{verr}
The variety $\widetilde X_O$ is smooth  and   the projection  $\widetilde X_O\to \P^2_W$ is a conic bundle with discriminant   a smooth sextic $\Gamma^\star_6\subset \P^2_W$.
\end{prop} 

 We denote the 
associated double cover by $\pi^\star:\widetilde \Gamma^\star_6\to\Gamma^\star_6$.  It is \'etale by \cite{bbb}, prop.~1.5, and connected.\footnote{\label{brauer}Any nontrivial conic bundle over $\P^2_O$ defines a 
  nonzero element of the Brauer group of $\C(\P^2_O)$  which, since the Brauer group  of $\P^2_O$ is trivial, must have some nontrivial residue: the double cover of at least one component of the discriminant curve must be irreducible.}
  
  Also,   fibers of $\widetilde X_O\to \P^2_W$ project to   conics in $X_O$. Since the strict transform of $\P^3_W$ in $\widetilde W_O$ is a $\P^1$-bundle (\S\ref{coowo}), they  meet  the quadric $Q$ in two points.

\begin{proof}We saw in \S\ref{p2w}
  that   $\widetilde W_O\to \P^2_W$  is a $\P^2$-bundle. The smoothness of  $\widetilde X_O$  then follows from the Bertini theorem, which also implies that the discriminant curve is smooth. The fact that it is a sextic follows either from direct calculations or from Theorem \ref{nodver}.
\end{proof}

With the notation of \S\ref{pmo},  consider now the rank-$3$ vector bundle $\cM_O\to \P^2_O$ whose fiber over $[V_4]$ is $M_{V_4}/\!\!\wedge^2\! V_O$. 
 Inside $\P(\cM_O)$, the   quadric $\Omega_O$ defines a   conic bundle $Q_{\Omega,O}\to \P^2_O$, and $Q_{\Omega,O}$ is smooth by the Bertini theorem. 
 
 \begin{prop}\label{verr2}
The double cover associated with the conic bundle $Q_{\Omega,O}\to \P^2_O$ is isomorphic to 
$\pi^\star:\widetilde \Gamma^\star_6\to\Gamma^\star_6$. 
\end{prop}

\begin{proof}We saw in \S\ref{p2w}
  that the  $\P^2$-bundles  $\widetilde W_O\to \P^2_W$  and
 $\P(\cM_O)\to \P^2_O$ are isomorphic and the isomorphism restricts to an isomorphism between the conic bundles   $\widetilde X_O\to \P^2_W$ and $Q_{\Omega,O}\to \P^2_O$, because they are   pull-backs of the same quadric $\Omega_O\subset \P^6_O$.  It follows that the associated double covers  are     isomorphic.
\end{proof}

   \subsection{The double cover $\widetilde \Gamma_6\to \Gamma_6$}\label{g6}

 Let $\P$ be the net of quadrics in $\P^6_O$ which contain $X_O$. The discriminant curve $\Gamma_7\subset\P$ (which parametrizes singular quadrics in $\P$) is the union of the line $\Gamma_1$  of quadrics that contain $W_O$, and  a   sextic $\Gamma_6$. The pencil $\Gamma_1$ meets $\Gamma_6$ transversely at six points $p_1,\dots,p_6$  corresponding to quadrics whose vertices are  the six  singular points $s_1,\dots,s_6$ of $X_O$  (\cite{bbb}, \S5.6.2). 

All quadrics in $\P$ have rank at least $6$ (because rank-$5$ quadrics in $\P^6_O$ have codimension $3$). 
 In particular,   there is a   double \'etale cover 
 $$\pi:\widetilde \Gamma_6\cup\Gamma_1^1\cup\Gamma_1^2\to \Gamma_6\cup\Gamma_1$$ 
 corresponding to the choice of a ($2$-dimensional) family of $3$-planes contained in a quadric of rank $6$ in $\P $. The $3$-plane $ \P^3_W$ is contained in all the quadrics in the pencil $\Gamma_1 $ and defines the component $\Gamma_1^1$  and points $\{\tilde p_1,\dots,\tilde p_6\}=\Gamma_1^1\cap\widetilde \Gamma_6$.
 The   curve $\widetilde \Gamma_6$ is smooth and  connected  (footnote \ref{brauer}).

 \subsection{Line transformations}\label{lt}
 Let $\ell$ be a general line contained in $X$. As in the smooth case (\cite{dim}, \S6.2), one can perform a {\em line transformation} of $X$ along $\ell$:  
 \begin{equation*} 
\xymatrix@C=35pt@M=7pt@R=20pt
{\widetilde X_\ell\ar[d]^\eps\ar@{-->}[r]^\chi&\widetilde X'_\ell\ar[d]^{\eps'}\\
X\ar@{-->}[r]^{\psi_\ell}&X_\ell,
}\end{equation*}
where $\eps$ is the   blow-up of  $\ell$, with exceptional divisor $E$, the birational map $\chi$ is an $(-E)$-flop, $\eps'$ is the blow-down onto a line $\ell'\subset X_\ell$ of a divisor $E'\equiv  -K_{\widetilde X'_\ell}-\chi(E)$, and 
$X_\ell$ is another nodal threefold of type $\cX_{10}$. The map $\psi_\ell$ is associated with a linear subsystem of $|\cI_\ell^3(2)|$. Its inverse $\psi_\ell^{-1}$ is the line transformation of $X_\ell$ along the line $\ell'$. Moreover,  $\psi_ \ell $ is defined at the node $O$ of $X$ and  $\psi_ \ell ^{-1}$ is defined at the node $O'$  of $X_ \ell $. 

 As explained in \cite{ip} \S4.1--4.3, this is a general process. One can also perform it with the image $\ell_O$  of $\ell$ in  $ X_O$ and obtain a diagram (\cite{ip}, Theorem 4.3.3.(ii))
\begin{equation*}\label{dia2}
\xymatrix@C=35pt@M=7pt@R=20pt
{\widetilde X_{O,\ell_O}\ar[d]\ar@{-->}[r]&\widetilde X'_{O,\ell_O}\ar[d]^{p'_\ell}\\
X_O\ar@{-->}[r]^{p_\ell}&\P^2,
}\end{equation*}
where $p'_\ell$ is a conic bundle and $p_\ell$ is again associated with a linear subsystem of $|\cI_{\ell_O}^3(2)|$. The  birational conic bundle $p_\ell$    can be described geometrically as follows (\cite{bbb}, \S1.4.4 or \S6.4.2): a general point $x\in X_O$ is mapped to the unique quadric
in $\P=\P^2$ containing the 2-plane $\langle \ell_O, x\rangle$. Its discriminant curve is $\Gamma_7=\Gamma_1\cup \Gamma_6$.

\begin{lemm}\label{lemp}
Let $X_\ell\dra X_{\ell,O'}\subset \P^6_{O'}$ be the projection from the node $O'$ of $X_\ell$ and let $\ell'_{O'}$ be the image of the line $\ell'\subset  X_\ell$.
There is a commutative diagram
\begin{equation*}
\xymatrix@C=35pt@M=7pt@R=20pt
{X_O\ar@{-->}[d]^{p_W}\ar@{-->}[r]^{\psi_{\ell,O}}&X_{\ell,O'}\ar@{-->}[d]^{p_{\ell'}}\\
\P^2_W\ar@{=}[r]&\P'.
}\end{equation*}
\end{lemm}

\begin{proof}
A general fiber   of $p_W$ is a conic which meets $Q$ in two points (\S\ref{pw}). Its strict transform $F$ on $X$ is a quartic curve that passes twice through the node $O$ and does not meet $\ell$, hence its image   $\psi_\ell(F)\subset X_\ell$ has degree 8 and passes twice through the node $O'$ of $X_\ell$. Moreover, 
\begin{eqnarray*}
 E'\cdot \chi(\eps^{-1}(F))  &=& ( -K_{\widetilde X'_\ell}-\chi(E))\cdot \chi(\eps^{-1}(F))\\
 &  =&( -K_{\widetilde X_\ell}-E) \cdot \eps^{-1}(F)\\
 &=& \eps^*\cO_X(1)(-2E)  \cdot\eps^{-1}(F)\\
 &=&\deg(F)=4.
 \end{eqnarray*}
The image of $F$ in $X_\ell$ is therefore an octic that passes twice through the node $O'$ and meets $\ell'$ in 4 points. 
Its image in $X'_{\ell,O'}$ is a rational sextic  that meets  $\ell'_{O'}$ in 4   points. Since the map  $p_{\ell'} $ is associated with a linear subsystem of $|\cI_{\ell'_{O'}}^3(2)|$, it contracts this sextic. In other words, images of general fibers of $p_W$ by $\psi_{\ell,O} $ are contracted by $p_{\ell'} $, which proves the lemma.\end{proof}

 \subsection{The Verra solid associated with $X$}\label{veso}
 A {\em Verra solid} (\cite{ver}) is a smooth hypersurface of bidegree $(2,2)$ in $\P^2\times \P^2$.
 
 \begin{theo}\label{nodver}
A general nodal Fano threefold of type $\cX_{10}$ is birational to a general Verra solid. 
\end{theo}

More precisely, let $X$ be a general nodal Fano threefold of type $\cX_{10}$.  We show that for a suitable choice of $\ell\subset X_O$, the two birational conic bundle structures:
\begin{itemize}
\item  
 $ p_W  : X \dashrightarrow \P_W^2 $, with discriminant curve $\Gamma^\star_6$ (\S\ref{pw}), and 
 \item  
 $ p_\ell  : X \dashrightarrow  \P $, with   discriminant curve $\Gamma_6\cup \Gamma_1$ (\S\ref{lt}), 
 \end{itemize}
 induce a birational isomorphism 
$$\psi_\ell:=(p_W,p_\ell) : X \dashrightarrow T,  $$
where $T\subset \P_W^2\times\P$ is a general Verra solid. In particular, the sextics $\Gamma_6$ and 
$\Gamma^\star_6$ are general.

\begin{proof}[Proof of the theorem]
Recall that $X_O=W_O\cap \Omega_O$ contains the smooth quadric surface $Q=\P^3_W\cap \Omega_O$. Instead of choosing a general line as in \S\ref{lt}, we choose a line   $\ell$ contained in $Q$, not passing through any of the six singular points of $X$.

In suitable coordinates,   $W_O\subset \P^6_O$ is  the intersection of  the quadrics 
\begin{eqnarray*}
\Omega_1(x) &= &x_0x_1+x_2x_3+x_4x_5 \\  
\Omega_2(x) &= &x_1x_2+x_3x_4+x_5x_6,
\end{eqnarray*}
and  $\P^3_W=\P(\langle e_0,e_2,e_4,e_6\rangle)$ (\S\ref{coowo}). 
We may   assume   $\ell=\P(\langle e_2,e_4\rangle)$. 
  
The quadric $\Omega_O$ contains $\ell$, hence its 
equation is of the form
$$\Omega_O(x)=x_2\lambda_2(x')+x_4\lambda_4(x')+q(x'),$$
where  $\lambda_2$ and $\lambda_4 $   are linear and $q$ quadratic in $x'=(x_0,x_1,x_3,x_5,x_6)$. 

The   projection 
$p_W:X_O\dra\P^2_W$  sends $x $ to $(x_1,x_3,x_5)$. 
The map $p_\ell:X_O\dra \P=\langle \Omega_1,\Omega_2,\Omega_O\rangle$ sends  a general point $x\in X_O$ to the unique quadric
in $\P$ containing the 2-plane $\langle \ell, x\rangle$ (\S\ref{lt}).  We obtain
$$p_\ell(x)=(x_1\lambda_4(x')-x_3\lambda_2(x'), x_3\lambda_4(x')-x_5\lambda_2(x'), x_3^2-x_1x_5).$$
A general (conic) fiber of $p_W$  is mapped by $p_\ell$   onto a conic in $\P$ because $p_\ell$  becomes   linear once restricted to a fiber of $p_W$.  

Similarly, let us consider the fiber of $p_\ell$ at the general point  $[\Omega_O]\in\P$. It is the
set of points $x\in X_O$ such that the $2$-plane $\langle \ell, x\rangle$ is contained in $\Omega_O$.
This means   $\lambda_2(x')= \lambda_4(x')=0$, and consequently $q(x')=0$, so that $x'$ describes 
a plane conic in its parameter space  $\P(\langle e_0,e_1,e_3,e_5,e_6\rangle)$. 
Since the   projection to $\P^2_W$ factors through this 4-plane, a general
fiber of $p_\ell$ is mapped by $p_W$ birationally onto a conic in $\P^2_W$.  

Since the restriction of $p_W$ (resp.~$p_\ell$) to a general fiber of $p_\ell$ (resp.~$p_W$)
is birational onto its image, the product map
$$\psi_\ell=(p_W,p_\ell) : X_O\dashrightarrow \P_W^2\times\P$$
 is birational onto its image $T\subset  \P_W^2\times\P$. 
Let $(d,e)$ be the bidegree of this hypersurface: $d$ (resp.~$e$) is the degree of the image by $p_W$ 
(resp.~$p_\ell$) of a general fiber of  $p_\ell$ (resp.~$p_W$). Since   these plane curves 
are conics, we have $d=e=2$, and $T$ is a Verra solid.

Finally, the fact that $T$ is general  
follows from a dimension count. 
\end{proof}

\begin{rema} Let  $\rho_1:T\to  \P_W^2$ and $\rho_2:T\to \P$ be the two projections. {\em The surface $\psi_\ell(Q)\subset T$ is isomorphic to the blow-up of $Q$ at the six singular points $s_1,\dots,s_6$ of $X_O$. It is also equal to $\rho_2^{-1}(\Gamma_1)$.}

For $x$   in $Q\moins \ell$, the only quadrics in $\P$ that contain 
  the 2-plane $\langle \ell, x\rangle$ are in the pencil $\Gamma_1$. This implies $p_\ell(Q)\subset \Gamma_1$ and 
  $\psi_\ell(Q)\subset \rho_2^{-1}(\Gamma_1)$. A dimension count shows that given the Verra solid $T\subset  \P_W^2\times\P$, the line $\Gamma_1\subset \P$ is general. It follows that     $\rho_2^{-1}(\Gamma_1)$ is a smooth del Pezzo surface of degree 2 which is equal to $\psi_\ell(Q)$; its anticanonical (finite) map is the double cover  $\rho_{1Q}:\psi_\ell(Q)\to \P_W^2$. 

With the notation of \S\ref{pw},  it follows from the comments at the end of \S\ref{p2w} that the inverse image $\widetilde Q$ of $Q $ in $\widetilde X_O$  is isomorphic to the blow-up of $Q$ at the six singular points of $X_O$. 
We have a commutative diagram
$$\xymatrix
{
\widetilde Q\ar@{-->}[rr]^{\psi_Q}
\ar[dr]_{p_W}&&\psi_\ell(Q)\ar[dl]^{\rho_{1Q}}\\
&\ \P_W^2.
}
$$
Since $\rho_{1Q}$ is finite, $\psi_Q$ must be a morphism; since $\widetilde Q$ and $\psi_\ell(Q)$ are both del Pezzo surfaces of degree 2, it is an isomorphism.
\end{rema}

\begin{rema}\label{newr}
Let us analyze more closely the conic bundle structure  $ p_\ell  : X \dashrightarrow  \P $ for our choice of $\ell\subset Q$.
 Following the classical construction of \cite{bbb}, 1.4.4, we set  
$$X_\ell=\{ (P,\Omega_p)\in G(2,\P^6_O)\times \P\mid \ell\subset P\subset \Omega_p\}$$
and consider the birational map
\begin{eqnarray*}
\phi_\ell:X&\dra &X_\ell\\
x&\longmapsto& (\langle \ell,x\rangle,p_\ell(x)).
\end{eqnarray*}
Away from $\Gamma_1$, the second projection $q_\ell:X_\ell\to \P$ is a conic bundle with discriminant curve $ \Gamma_6$, whereas $q_\ell^{-1}(\Gamma_1)$ is the union of two components:
$$ Q_1=\{  (P,\Omega_p) \in G(2,\P^6_O)\times \Gamma_1\mid \ell\subset P\subset \P^3_W\} \isom \P^1\times\Gamma_1, $$
and the closure of
$$ \{  (P,\Omega_p) \in G(2,\P^6_O)\times \Gamma_1\mid  P\cap \P^3_W=\ell,\ P\subset \Omega_p \}.$$
They correspond respectively to the two components $\Gamma_1^1$ and $\Gamma_1^2$ of $\pi^{-1}(\Gamma_1)$ (see \S\ref{g6}). We have $\phi_\ell(Q)=Q_1$ and   lines $\ell^-\subset Q$ that meet  $\ell$ map to sections  of $Q_1\to\Gamma_1$. Consider the diagram
$$\xymatrix@C=50pt@M=7pt@R=-10pt
{
Q_1&Q\ar[l]_-{\sim}&\psi_\ell(Q)\ar[l]_\eps \\
\cap&\cap&\cap\\
X_\ell\ar[ddddddddr]_{q_\ell}&X\ar@{-->}[l]_-{\phi_\ell}\ar@{-->}[r]^{\psi_\ell}\ar[dddddddd]_{p_\ell}&T\ar[ddddddddl]^{\rho_2}\\
\\\\\\\\\\\\\\
&\ \P.
}
$$
The map $\rho_{2Q}:\phi_\ell(Q)\to \Gamma_1$ has six reducible fibers, above the points $p_1,\dots,p_6$  of $\Gamma_1\cap \Gamma_6$, and each contains one exceptional divisor  $E_1,\dots,E_6$ of $\eps$. Since lines $\ell^+\subset Q$ from the same ruling as   $\ell$ map to fibers of $\rho_{2Q}$, the other components must be  the strict transforms  $\ell^+_1,\dots,\ell^+_6$ of these lines  passing through $s_1,\dots,s_6$.

As   mentioned above, a general line $\ell^-\subset Q$ from the other ruling maps to a section of $\rho_{2Q}$ that does not meet  $E_1,\dots,E_6$, hence must meet $\ell^+_1,\dots,\ell^+_6$; moreover, these components of reducible fibers of the conic bundle $\rho_2:T\to\P$ correspond to the points of $\Gamma_1^1\cap \widetilde\Gamma_6$, which we denoted by $\tilde p_i$ in \S\ref{g6}. We also denote by $\ell^-_1,\dots,\ell^-_6$ 
the strict transforms  of those lines   passing through $s_1,\dots,s_6$. The line $\ell^-_i$ meets $\ell^+_j$ if and only if $i\ne j$, and meets $E_j$ if and only if $i=j$. It   maps by both projections $\rho_1$ and $\rho_2$ to a line. Finally, $\ell^\pm_i+E_i$ is rationally equivalent to $\ell^\pm$.
\end{rema}

 \section{The variety of conics contained in $X$}

 \subsection{The surfaces $F_g(X)$ and $F(X)$}\label{fgf}
 
 We follow \cite{lo}, \S5, but  with the notation of \cite{dim}, \S5. In particular, $F_g(X)$ is the variety of conics\footnote{As remarked in \cite{dim}, \S3.1, any nonreduced conic contained in $W$ is contained in a 2-plane contained in $W$. Since the family of these 2-planes has dimension 1 and none of them contain $O$, and nonreduced conics have codimension 3, all conics contained in a general $X$ are reduced.}
  contained in $X$, and
  $$F(X)=\{(c,[V_4]) \in F_g(X)\times \P(V_5^\vee) \mid  c\subset G(2,V_4)\}.
$$
The projection $F(X)\to F_g(X)$ is an isomorphism except over the one point corresponding to the only $\rho$-conic $c_X$ contained in $X$ (\cite{dim}, \S5.1). 

We define as in \cite{dim}, \S5.2 an involution $\iota$ on $F(X)$ as follows.
  For any hyperplane $V_4\subset V_5$,  define quadric surfaces
  $$Q_{W,V_4}=G(2,V_4)\cap \P(M_{V_4})\quad{\rm and}\quad Q_{\Omega,V_4}= \Omega\cap \P(M_{V_4}).
  $$
   If $(c,[V_4]) \in F(X)$, the intersection  
   $$ X\cap \P(M_{V_4})=Q_{W,V_4}\cap Q_{\Omega,V_4},
$$
has dimension 1.\footnote{This is because, as we saw in \S\ref{pox}, $X$ is locally factorial and  $\Pic(X)$ is generated by $\cO_X(1)$, hence the degree of any surface contained in $X$ is divisible by 10.} Since $c$ is reduced, one checks by direct calculation that as a 1-cycle, it is the sum of 
 $c$ and another (reduced) conic contained in $X$, which we denote by $\iota(c)$. One checks as in \cite{lo}, Lemma 3.7, that since $X$ is general, $\iota(c)\ne c$ for all $c$, and some quadric   in the pencil $\langle Q_{W,V_4}, Q_{\Omega,V_4}\rangle$ is a pair of  distinct  planes. This defines a fixed-point-free involution $\iota$ on $F(X)$.

\subsection{Conics in $X$ passing through $O$}\label{cto}  Since $O\in O_4$, any such conic $c$   is a $\tau$-conic, hence is   contained in a unique $G(2,V_4)$, and $[V_4]\in \P^2_O$. The quadric  $Q_{\Omega,V_4}$ is then a cone with vertex $O$.
 
 If   $Q_{\Omega,V_4}$ is reducible, \ie, if $[V_4]\in \Gamma^\star_6$, the conics  $c$ and $\iota(c)$  meet at $O$ and another point. The two points  of $\pi^{\star -1}([V_4])$ correspond to the two 2-planes contained in   $Q_{\Omega,V_4}$, hence to $[c]$ and $[\iota(c)]$. By Proposition \ref{verr2}, these conics $c$  are parametrized by the 
  curve $\widetilde\Gamma^\star_6$.

  If   $Q_{\Omega,V_4}$ is irreducible,  $c\cup \iota(c)$ is the intersection of the cone $Q_{\Omega,V_4}$ with a pair of planes, and  $c$ is the union of two (among the six) lines in $X$ through $O$.
  


  \begin{theo}[\cite{lo}, \S5]\label{fg}
  The variety $F_g(X)$ is an irreducible surface. Its singular locus is the smooth connected curve $\widetilde\Gamma^\star_6$ of conics on $X$ passing through $O$ described above, and its  normalization
  $\widetilde F_g(X)$  is smooth. 
    \end{theo}
    
    Moreover, the curve of $\sigma$-conics (which is disjoint from $\widetilde\Gamma^\star_6$) is exceptional on $F_g(X)$, and its inverse image
     on     $\widetilde F_g(X)$ can be contracted  to a smooth surface $ \widetilde F_m(X)$ (as in the smooth case; see \cite{dim}, \S5.3). The involution $\iota$ induces  an involution on $ \widetilde F_m(X)$.

\medskip
Logachev also proves that the inverse image of $\widetilde\Gamma^\star_6$ in $ \widetilde F_g(X)$  has two connected components $\widetilde\Gamma^\star_{6,+}$ and $ \widetilde\Gamma^\star_{6,-}$, which map isomorphically to $\widetilde\Gamma^\star_6$ by the normalization $\nu:\widetilde F_g(X)\to   F_g(X)$. They can be described as follows.  

Let $c\subset X$ be a   conic passing through $O$ corresponding to a point of $ \widetilde \Gamma_6^\star $. The curve   $p_O(c)$  is a line $\ell\subset X_O$ that meets, but is not contained in, the smooth quadric surface $Q=\P^3_W\cap \Omega_O$. The two points of $\nu^{-1}([c])$ correspond to the conics $[\ell\cup \ell^+]\in \widetilde\Gamma^\star_{6,+}$ and $[\ell\cup \ell^-]\in \widetilde\Gamma^\star_{6,-}$ contained in $X_O$, where $\ell^+$ and $\ell^-$ are the two lines in $Q$ passing through its point of intersection with $\ell$.

 In particular, $\widetilde\Gamma^\star_{6,\pm}$ carries an involution $\sigma^\star_\pm$ induced by the involution $\sigma^\star$ of $\widetilde\Gamma^\star_6$.
On the other hand, the involution $\iota$ of $\widetilde F_g(X)$ maps $\widetilde\Gamma^\star_{6,+}$ isomorphically onto $\widetilde\Gamma^\star_{6,-}$. The identification  $\widetilde\Gamma^\star_{6,+}\isomto \widetilde\Gamma^\star_{6,-}$ induced by the normalization $\nu $ is $\iota\circ \sigma^\star_+=\sigma^\star_-\circ\iota$.

  \subsection{The special surfaces $S^{\rm even}$ and $S^{\rm odd}$}\label{seso}

There is an embedding $\P^\vee\hookrightarrow \Gamma_6^{(6)}$ that sends a line in $\P$ to its intersection with $\Gamma_6$. Its inverse image in $\widetilde \Gamma_6^{(6)}$ is a surface  $S$  with two connected   components $S^{\rm even}$ and $S^{\rm odd}$,  each endowed with an involution $\sigma$. They are defined by
 \begin{equation*}\label{evenodd}
S^{\rm even} = \{[\widetilde D]\in   S\mid h^0(\widetilde\Gamma_6, \pi^*\cO_\P (1) (\widetilde D)) \hbox{\rm\ even}\}
\end{equation*}
and similarly for $ S^{\rm odd}$ (\cite{spe}, \S2, cor.). By 
 \cite{spe}, prop.~3,   they are smooth because $\Gamma_6$, being general,  has no tritangent lines.

In particular, the point  $\tilde p_1+\dots+\tilde p_6$ of $\widetilde \Gamma_6^{(6)} $ defined at the end of \S\ref{g6} is in $  S$. {\em We will show in Proposition \ref{MS} that it is in $S^{\rm odd}$.}

  \subsection{The isomorphism  $\widetilde F_m(X)\isomto   S^{\rm odd}$}\label{s117} 
Let $c$ be a conic on $X$ such that  $O\notin\langle c\rangle$. The projection $p_O(c)$ is a conic in $X_O$, and the set of quadrics in $\P$ that contain $\langle p_O(c)\rangle$ is a line $L_c\subset \P$.  For each point $p$ of $L_c\cap \Gamma_6$, {\em if the vertex $v_p$ of $\Omega_p$ is not in the $2$-plane} $\langle p_O(c)\rangle$, the $3$-plane $\langle p_O(c),v_p\rangle$ defines a point $\tilde p\in\widetilde \Gamma_6$ above $p$. This defines a point $\rho_g([c])$ in $  S $.

One checks by direct calculation (\S\ref{coowo}) that the $2$-plane   $\Pi=\langle c_X \rangle $ is disjoint from $\T_{W,O}$. It follows that $c_X$ satisfies the conditions above, hence $\rho_g([c_X])$ is well-defined. Moreover, the line $L_{c_X}$ is $\Gamma_1$, and for each $p_i\in \Gamma_1\cap \Gamma_6$,  the $3$-planes $ \P^3_W$ and $\langle p_O(c_X),v_{p_i}\rangle$ meet only at $v_{p_i}$, hence belong to different families; it follows that we have $\rho_g([c_X])=\sigma (\tilde p_1+\dots+ \tilde p_6)$, {\em which is in the surface $S^{\rm odd}$}.
We have therefore defined a rational map 
$$\rho_g:  F_g(X)\dra S^{\rm odd}
.$$
Logachev then proves (\cite{lo}, \S5) that $\rho_g$ induces an isomorphism 
  \begin{equation}\label{rho}\rho:\widetilde F_m(X)\isomto   S^{\rm odd}.
    \end{equation}
 This isomorphism commutes with the involutions $\iota$ and $\sigma$: since   the $3$-planes $\langle p_O(c),v_p\rangle$ and $\langle p_O(\iota(c)),v_p\rangle$ meet in codimension $1$, they belong to different families, hence 
  $\rho_g\circ \iota=\sigma\circ \rho_g$.

Let us explain how $\rho$ is defined on the normalization $\widetilde F_g(X)$.
If $c\subset X$ is a   conic passing through $O$ corresponding to a point of $ \widetilde \Gamma_6^\star $  and  $\ell=p_O(c)$, the two points of $\nu^{-1}([c])$ correspond to the conics $[\ell\cup \ell^+]$ and $[\ell\cup \ell^-]$, where $\ell^+$ and $\ell^-$ are the two lines in $Q$ passing through its point of intersection with $\ell$.

 The images by $\rho$ of these two points  are defined as usual:  the set of quadrics in $\P$ that contain $\langle \ell\cup \ell^\pm\rangle$ is a line $L^\pm\subset \P$; for each point $p$ of $L^\pm\cap \Gamma_6$,   if the vertex $v_p$ of $\Omega_p$ is not in the $2$-plane  $\langle \ell\cup \ell^\pm\rangle$, the $3$-plane $\langle \ell\cup \ell^\pm,v_p\rangle$ defines a point $\tilde p\in\widetilde \Gamma_6$ above $p$.   
 

\subsection{Conic transformations}\label{ct}
Let $c$ be a general conic contained in $X$. As in the smooth case (\cite{dim}, \S6.2),  one can perform a {\em conic transformation} of $X$ along $c$:  
\begin{equation*} 
\xymatrix@C=35pt@M=7pt@R=20pt
{\widetilde X_c\ar[d]^\eps\ar@{-->}[r]&\widetilde X'_c\ar[d]^{\eps'}\\
X\ar@{-->}[r]^{\psi_c}&X_c,
}\end{equation*}
where $\eps$ is the   blow-up of  $c$, the birational map $\chi$ is a flop, $\eps'$ is the blow-down onto a conic $c'\subset X_c$ of a divisor,   and 
$X_c$ is another nodal threefold of type $\cX_{10}$. Moreover,  $\psi_c$ is defined at the node $O$ of $X$ and  $\psi_c^{-1}$ is defined at the node  of $X_c$.

As in the smooth case, there is an isomorphism 
$$\phi_c: F_m(X)\isomto F_m(X_c)$$ which commutes with the involutions $\iota$ and $\phi_c(\iota([c]))=[c_{X_c}]$ (\cite{dim}, Proposition 6.2).

 \section{Reconstructing the threefold $X$}
 
We keep the notation of \S\ref{g6}:  in the net $\P$ of quadrics in $\P^6_O$ which contain $X_O$, the discriminant curve $\Gamma_7=\Gamma_6\cup\Gamma_1$ parametrizes singular quadrics, and the   double \'etale cover $\pi:\widetilde \Gamma_7\to \Gamma_7$ corresponds to the choice of a  family of $3$-planes contained in a quadric of rank $6$ in $\P^6_O$.


\begin{theo}\label{tcM}
We have the following properties.
\begin{itemize}
\item[\rm a)]The   morphism 
$ v:\Gamma_7\to \P^6_O$ that sends $p\in \Gamma_7$ to the unique singular point of the corresponding singular quadric $\Omega_p\subset\P^6_O$   is an embedding and
 $v^*:H^0(\P^6_O,\cO_{\P^6_O}(1))\to H^0(\Gamma_7,v^*\cO_{\P^6_O}(1))$ is an isomorphism.
Furthermore, the invertible sheaf $M_X$  on $\Gamma_7$ defined   by $M_X(1)=v^*\cO_{\P^6_O}(1)$  is a  theta-characteristic
 and   $H^0(\Gamma_7,M_X)=0$.
\item[\rm b)]
The double \'etale cover $\pi:\widetilde \Gamma_7\to \Gamma_7$ is defined by the point $\eta=M_X(-2)  $, of order $2$  in $J(\Gamma_7)$.
\item[\rm c)]
The variety $X_O\subset \P^6_O$ is determined up to projective isomorphism by the pair $(\Gamma_7,M_X)$.\footnote{\label{more}More precisely, given an isomorphism $f:\Gamma_7\isomto  \Gamma'_7$ such that $f^*M_{X'}= M_X$, there exists a projective isomorphism $X_O \isomto   X'_O$ which induces $f$.}
\end{itemize}
\end{theo}

\begin{proof} Item a) is proved in the same way as \cite{bbb}, lemme 6.8 and lemme 6.12.(ii); item b)  as \cite{bbb},  lemme 6.14; and item c)   as \cite{bbb}, prop.~6.19. The isomorphism  $v^*\cO_{\P^6_O}(2)\isom \cO_{\Gamma_7} (6)$ can also be seen directly by noting that a quadric in the net $\P$ is given by a $7\times7$ symmetric matrix $A$ of linear forms on $\P$, and that when this matrix has rank 6, the comatrix of $A$ (whose entries are sextics) is of the type $(v_iv_j)_{1\le i,j\le 7}$, where $v_1,\dots,v_7$ are the homogeneous coordinates of the vertex.
 \end{proof}
 

  \begin{rema}\label{conv}
 Conversely, given a reduced  septic $\Gamma_7\subset\P$  and an invertible sheaf $M$ on $\Gamma_7$ that satisfies $H^0(\Gamma_7,M)=0$ and $M^2\isom \cO_{\Gamma_7} (4)$, there is a resolution 
\begin{equation}\label{resol}
0\lra \cO_{\P}(-2)^{\oplus 7}\stackrel{A}{\lra}  \cO_{\P}(-1)^{\oplus 7}\lra M\lra 0
\end{equation}
of $M$ viewed as a sheaf on $\P$, where $A$ is a $7\times7$ symmetric matrix of linear forms, everywhere of rank $\ge 6$,
 with determinant an equation of $\Gamma_7$. Indeed, this follows from work of Catanese (\cite{cata}, Remark 2.29 and  \cite{cc}, Theorem 2.3) and Beauville (\cite{full}, Corollary 2.4; the hypothesis that $\Gamma_7$ be integral is not needed in the proof because $M$ is invertible) who generalized
 an old result of  Dixon's (\cite{dix})
 for smooth plane curves. 
 
Given any smooth sextic $ \Gamma_6\subset\P$ that meets a line $\Gamma_1$ transversely, there exists on the septic $\Gamma_7=\Gamma_6\cup\Gamma_1 $  an 
invertible theta-characteristic   $M$ such that $H^0(\Gamma_7,M)=0$ (\cite{catas}, Theorem 7 or Proposition 13). So we  obtain by (\ref{resol}) a net of quadrics in $\P^6_O$ whose base-locus is the intersection of $W_O$ with a smooth quadric. By taking its inverse image under the birational map $W\dra W_O$, we obtain a variety of type $\cX_{10}$ with, in general, a single node at $O$.
 \end{rema}
 
 The following proposition
 describes how the theta-characteristic $M_X$ is related to our previous constructions.
 
\begin{prop}\label{MS}
Let   $\Gamma_6$ be a general plane sextic, let $\pi:\widetilde \Gamma_6\to \Gamma_6$ be a connected  double \'etale cover, with associated involution $\sigma$ and line bundle  $\eta$  of order $2$ on $\Gamma_6$. There is a commutative diagram
$$\xymatrix
{
 {\left\{ 
\begin{matrix}
\hbox{Invertible theta-characteristics}\\
\hbox{$M$ on the union of $\Gamma_6$ and a}\\
\hbox{transverse line such that $M\vert_{\Gamma_6}\isom \eta(2)$}
\end{matrix}
\right\}}
\ar[rr]^-\theta\ar[dr]&&  S/\sigma\ar[dl]\\
&\P^\vee,
}
$$
where $\theta$ is   an open embedding   and maps {\em even} (resp.~{\em odd}) theta-char\-act\-eristics to $S^{\rm odd}/\sigma$
(resp.~$S^{\rm even}/\sigma$).

Furthermore, if $M_X$ is the (even) theta-characteristic associated with  a general nodal $X$, we have
$$\theta(M_X)= \rho([c_X]).$$
\end{prop}

(The map $\rho$ was defined in \S\ref{s117}.)

\begin{proof}
For any invertible theta-characteristic $M$ on $\Gamma_7=\Gamma_6\cup\Gamma_1\subset \P$ such that $M\vert_{\Gamma_6}\isom \eta(2)$, the invertible sheaf $M(-2)$ has order $2$ in  $J(\Gamma_7)$, hence defines a   double \'etale cover $\pi:\widetilde \Gamma_7\to \Gamma_7$ which induces the given cover over $\Gamma_6$. The inverse image of $\Gamma_1$ splits as the disjoint union of two rational curves $\Gamma_1^1$ and $\Gamma_1^2$, and the intersections $\Gamma_1^i\cap\widetilde \Gamma_6$ 
 define divisors $\widetilde D =\tilde p_1+\dots+\tilde p_6$ and $\sigma^*\widetilde D $, hence a well-defined point  in  $  S/\sigma$.
 
 This defines a morphism $\theta$ which is compatible with the morphisms to $\P^\vee$. These morphisms are both finite of degree $2^5$ (\cite{har}, Theorem 2.14), hence $\theta$ is an open embedding. Since $S/\sigma$ is smooth with two components, the set of $M$ as in the statement of the theorem, with fixed parity, is smooth and irreducible.
 
 Finally, for any $s\in H^0(\widetilde \Gamma_6,\pi^*\cO(1)(\widetilde D))$, set
$$s^\pm=\sigma^*s\cdot s_{ \widetilde D}\pm s\cdot s_{\sigma^*\widetilde D }.
$$
We have
$$s^-\in H^0(\widetilde \Gamma_6,\pi^*\cO(2))^-\isom H^0( \Gamma_6, M\vert_{\Gamma_6})$$
and
$$s^+\in H^0(\widetilde \Gamma_6,\pi^*\cO(2))^+\isom H^0(\Gamma_6,  \cO(2)) \isom H^0( \P, \cO(2)).$$
Since the points $p_i=\pi(\tilde p_i)$, $i\in\{1,\dots,6\}$, are distinct, we have an   exact sequence
\begin{equation}\label{hhh}
0 \to H^0(\Gamma_7 ,M) \to H^0(\Gamma_6 ,M\vert_{\Gamma_6})\oplus H^0(\Gamma_1 ,M\vert_{\Gamma_1}) \to \bigoplus_{i=1}^6\C_{p_i},
\end{equation}
and since $s^+(p_i)=s^-(p_i)$, the pair $(s^-,s^+\vert_{\Gamma_1})$ defines an element of $H^0(\Gamma_7 ,M)$. Since $s^-=0$ if and only if $s\in s_{ \widetilde D}\cdot H^0( \widetilde \Gamma_6,\pi^*\cO(1) )$, we obtain an exact sequence
\begin{equation}\label{sexa}
0\to H^0( \Gamma_6,\cO(1) ) \to H^0(\widetilde \Gamma_6,\pi^*\cO(1)(\widetilde D))\to
H^0(\Gamma_7 ,M).
\end{equation}
Since $\Gamma_6$ is general, there is, by \cite{catas}, Theorem 7 or Proposition 13, a tranverse line $\Gamma_1$ and an even theta-characteristic   $M$ on $\Gamma_7=\Gamma_6\cup\Gamma_1$ such that $H^0(\Gamma_7,M)=0$. Because of (\ref{sexa}),  $\theta(M)$ is    in $S^{\rm odd}/\sigma$; we proved above that the set of all even  theta-characteristics is irreducible, hence it must map to $S^{\rm odd}/\sigma$, and odd 
  theta-characteristics must map to $S^{\rm even}/\sigma$.
\end{proof}
 
 \begin{coro} \label{reco}
A general nodal $X$ of type $\cX_{10}$ can be reconstructed, up to projective isomorphism, from the double \'etale cover $\pi:\widetilde \Gamma_6\to \Gamma_6$  and the point 
$\rho([c_X])$ of $S^{\rm odd}/\sigma$.
\end{coro}

\begin{rema}\label{proper}
The map $\theta$ is never an isomorphism: even if we allow lines $[\Gamma_1]\in\P^\vee$ which are (simply)  tangent to $\Gamma_6$ at a point $p_1$ (this happens when the quadric $\Omega_O$ is tangent to the cubic curve $C_O$), we only get  elements of $S/\sigma$ of the type $2\tilde p_1+\tilde p_3+\dots+\tilde p_6$. 

To fill up the remaining  16 points $ \tilde p_1+\sigma(\tilde p_1)+\tilde p_3+\dots+\tilde p_6$ of $S/\sigma$ above $[\Gamma_1]$, we need to let our nodal threefold ``degenerate'' to an $X$ with a node $O$ in the orbit $O_3$ (the pencil of quadrics that defines $W_O$ then contains a unique quadric of rank 5 and $\Gamma_6$ becomes tangent to $\Gamma_1$ at the corresponding point).
  Similarly, bitangent  lines $\Gamma_1$ to $\Gamma_6$ correspond to the case where the node $O$ is in the orbit $O_2$ (there are then two quadrics of rank 5 in the pencil $\Gamma_1$). Although  what should be done for flex lines is not clear, it seems likely that for $\Gamma_6$ general plane sextic, there should exist a  family of nodal $X$ parametrized by a suitable {\em proper} family of (even, possibly noninvertible) theta-characteristics on the union of $\Gamma_6$ and {\em any} line, isomorphic over $\P^\vee$ to  $S^{\rm odd}/\sigma$. This would fit with the results of \cite{dim}, where a   proper surface  contained in a general fiber of the period map is constructed: as explained in \S\ref{epm}, this surface degenerates in the nodal case to $S^{\rm odd}/\sigma$.
 \end{rema}

 \subsection{More on Verra solids} \label{vs}
 
  Let $\Pi_1$ and $\Pi_2$ be two copies of $\P^2$. Recall from \S\ref{veso} that a Verra solid is a  general  (smooth) hypersurface $T\subset \Pi_1\times\Pi_2$ of bidegree $(2,2)$. Each projection $\rho_i:T\to \Pi_i$ makes it into a conic bundle with discriminant curve  a smooth plane sextic $\Gamma_{6,i}\subset\Pi_i$ and associated  connected double \'etale covering $\pi_i:\widetilde \Gamma_{6,i}\to \Gamma_{6,i}$. We have
$$J(T)\isom  \Prym(\widetilde \Gamma_{6,1}/\Gamma_{6,1})\isom  \Prym(\widetilde \Gamma_{6,2}/\Gamma_{6,2}).
$$

The special subvariety $  S_i$ associated with the linear system $|\cO_{\Pi_i}(1)|$ is the union of two smooth connected surfaces $S^{\rm odd}_i$ and $S^{\rm even}_i$ (\S\ref{seso}).

Let $\cC$ the subscheme of the Hilbert scheme of $T$ that parametrizes   reduced connected purely $1$-dimensional subschemes of $T$ of degree $1$ with respect to both $\rho_1^*\cO_{\Pi_1}(1)$  and  $\rho_2^*\cO_{\Pi_2}(1)$.
For $T$ general, the scheme $\cC$ is a smooth   surface and a general element of each irreducible component corresponds to a smooth irreducible curve in $T$ (\cite{ver}, (6.11)). 

\begin{prop}\label{p115}
For $T$ general, the surfaces $\cC$, $S^{\rm even}_1$, and $S^{\rm even}_2$ are smooth, irreducible, and isomorphic.
\end{prop}

\begin{proof}
Let $[c]$ be a general element of $ \cC$. Each projection $\rho_i\vert_c :c\to  \Pi_i$ induces an isomorphism onto a line that meets $\Gamma_{6,i}$ in six distinct points. For each of these points $p$, the curve $c$ meets exactly one of the components of $\rho_i^{-1}(p)$, hence   defines an element of 
 $  S_i$.
This defines a rational map
$$ \cC\dra   S_i $$ over $\Pi^\vee_i$.

Conversely, let $L_i\subset \Pi_i$ be a general line. By Bertini, the surface $F_1=\rho_1^{-1}(L_1)$ is smooth and connected. It is ruled over $L_1$ with exactly six reducible fibers. The projection $\rho_2\vert_{F_1}:F_1\to\Pi_2$ is a double cover ramified along a smooth quartic $\Gamma_4\subset\Pi_2$, and $K_{F_1}\equiv -\rho_2^* L_2$. Let $L$ be a bitangent to $\Gamma_4$. The curve $\rho_2^{-1}(L)$ has two irreducible components and total degree $2$ over $\Pi_1$. Either one component is contracted by $\rho_1$ and it is then one of the $12$ components of the reducible fibers of $\rho_1\vert_{F_1}:F_1\to L_1$, or both components are sections and belong to $\cC$. Since there are $28$ bitangents, the degree of $\cC\to \Pi^\vee_1$ is $2\times (28-12)=32$.

More generally, for {\em any} line $L_1\subset \Pi_1$,  the surface $F_1=\rho_1^{-1}(L_1)$ is irreducible,  maps 2-to-1 to $\Pi_2$ with ramification a quartic, and only  finitely many curves are contracted. Any smooth   $[c]\in \cC$ with $c\subset F_1$ must map to a line everywhere tangent to the ramification, and there is only a  finite  (nonzero) number of such lines. 

It follows that every component of  $\cC$ dominates $ \Pi^\vee_1$. Moreover,   the morphisms    $\cC\to \Pi^\vee_i$, $ S^{\rm odd}_i\to \Pi^\vee_i$, and  $S^{\rm even}_i\to\Pi^\vee_i$ are all finite of degree $32$. It follows that the smooth surface $\cC$ is irreducible and maps isomorphically to a component of $  S_i $. It remains to prove that this component is  $S^{\rm even}_i$.

According to Theorem \ref{nodver}, we may assume that $T$ is obtained from a nodal $X$ of type $\cX_{10}$, as explained in \S\ref{veso}. In Remark \ref{newr}, we constructed  a  curve  $\ell^-_1$ of bidegree $(1,1)$ whose image in $S$ is $\sigma(\tilde p_1)+\tilde p_2+\dots+\tilde p_6$. Since
$\rho ([c_X])=\sigma (\tilde p_1+\dots+ \tilde p_6)$ (\S\ref{s117}) and $\rho ([c_X])\in S^{\rm odd}$
(Proposition \ref{MS}), this finishes the proof.
\end{proof}


\section{The extended period map}\label{anoth}

\subsection{Intermediate Jacobians}\label{ijac}
The intermediate Jacobian 
$J(X) $ of our nodal     $X$ appears as an extension
\begin{equation}\label{exte}
1\to\C^\star\to J(X)\to J(\widetilde X)\to 0,
\end{equation}
with extension class $e_X\in J(\widetilde X)/\pm1$. Since $\tilde X$ is birationally isomorphic to a general Verra solid $T$ (Theorem \ref{nodver}),   $J(\widetilde X)$ and $J(T)$, having no factors that are Jacobians of curves,   are isomorphic  (see the classical argument used for example in the proof of \cite{dim}, Corollary 7.6). In particular, by \cite{ver}, we have
\begin{equation}\label{g66}
J(\widetilde X)\isom \Prym(\widetilde \Gamma_6/\Gamma_6) \isom  \Prym(\widetilde \Gamma^\star_6/\Gamma^\star_6).\end{equation}

 
 \subsection{The extension class $e_X$}\label{extex}
Choose a general line $\ell\subset X_O$. A point  of $\widetilde \Gamma_6$ corresponds to a family of 3-planes contained in a singular $\Omega_p$. In this family, there is a unique 3-plane that contains $\ell$, and its intersection with $X_O$ is the union of $\ell$ and a rational normal cubic meeting $\ell$ at two points. So we get a family of curves on $\tilde X$ parametrized by $\widetilde \Gamma_6$ hence an Abel-Jacobi map
$$J(\widetilde \Gamma_6)\to J(\widetilde X)
$$
  which vanishes on $\pi^*J(\Gamma_6)$ and induces (the inverse of) the isomorphism (\ref{g66}) (here, $\Prym(\widetilde \Gamma_6/\Gamma_6)$ is seen as the quotient $J(\widetilde \Gamma_6)/ \pi^*J(\Gamma_6)$). Also, there is a natural map
$$\beta:S\to    \Prym(\widetilde \Gamma_6/\Gamma_6)\subset J(\widetilde \Gamma_6) $$
 defined up to translation (here, $\Prym(\widetilde \Gamma_6/\Gamma_6)$ is seen as the kernel of the norm morphism  $\Nm : J(\widetilde \Gamma_6)\to J(\Gamma_6)$)  which can be checked to be a closed embedding (as in \cite{spe}, C, p.~374). Finally, we have
  an Abel-Jacobi map
$$\alpha:\widetilde F_m(X) \to J(\widetilde X)
$$
(also defined up to translation). Logachev proves\footnote{This is \cite{lo}, Proposition 5.16, although the leftmost map in the diagram there should be $\Phi$, not $2\Phi$.}  that the   diagram
$$\xymatrix
{
\widetilde F_m(X)  \ar[r]^-\sim_-\rho \ar[d]^{\alpha}& S^{\rm odd}\ar[d]^\beta\\
J(\widetilde X)&\Prym(\widetilde \Gamma_6/\Gamma_6) \ar[l]_-\sim 
}
$$
commutes  up to a translation. In particular, $\alpha$ is a closed embedding.

 The extension class $e_X$ of (\ref{exte}) is the  image in $J(\widetilde X)/\pm 1$ by the Abel-Jacobi map of the difference between the (homologous) lines $\ell^+$ and $\ell^-$ from the two rulings of the smooth quadric surface $Q =\P^3_W\cap\Omega_O$  (\cite{col}, Theorem (0.4)).
 In particular, if $\nu: \widetilde F_m(X)\to F_m(X)$ is the normalization, it follows from \S\ref{cto} that we have for all $[c]\in \widetilde \Gamma_6^\star$
 $$e_X=\alpha([c^+])-\alpha([c^-]),
 $$
where $\nu^{-1}([c])=\{[c^+],[c^-]\}$. Since $\alpha$ is injective, $e_X$ {\em is  nonzero.}

 Note that $\alpha\circ\iota+\alpha$ is constant on $\widetilde F_m(X) $, say equal to $C$. We have 
\begin{equation}\label{tra}
\alpha(\widetilde\Gamma^\star_{6,-}) = \alpha(\widetilde\Gamma^\star_{6,+}) -e_X,
 \end{equation}
the involution $\sigma^\star_\pm$ on $\alpha(\widetilde\Gamma^\star_{6,\pm})$ is given by $x\mapsto C\pm e_X-x$, and 
these curves are Prym-canonically embedded   in $J(\widetilde X)\isom  \Prym(\widetilde \Gamma^\star_6/\Gamma^\star_6)$.

\begin{rema}\label{compact}
 The semi-abelian variety $J(X)$   can be seen as the complement in the $\P^1$-bundle $\P(\cO_{J(\widetilde X)}\oplus P_{e_X})\to J(\widetilde X)$ (where $P_{e_X}$ is the algebraically trivial line bundle on $J(\widetilde X)$  associated with   $e_X$) of the two canonical sections $J(\widetilde X)_0$ and $J(\widetilde X)_\infty$. A   (nonnormal) compactification $\overline{J(X)}$  is obtained by glueing these two sections by the translation $x\mapsto x+e_X$ and it is the proper limit of intermediate Jacobians of smooth threefolds of type $\cX_{10}$. The Abel-Jacobi map
 $ F_g(X) \dra J(X) $ 
then defines an embedding  (use Logachev's description of the nonnormal surface$ F_m(X)$ given in \S\ref{cto} and (\ref{tra}))
$$ F_m(X) \hookrightarrow  \overline{J(X)}. $$
 \end{rema}

\subsection{Extended period map}\label{epm}
Let  $\cN^\star_{10}$ (resp.~$ \partial \cN^\star_{10}$) be the stack of smooth or nodal threefolds of type $\cX_{10}$ (resp.~of nodal threefolds of type $\cX_{10}$), and let $ \cA^\star_{10}$ (resp.~$ \partial \cA^\star_{10}$) be the stack of \ppavs\ of dimension 10 and their rank-1 degenerations (resp.~of rank-1 degenerations) (\cite{mum}). There is an extended period map $\wp^\star:\cN^\star_{10}\to \cA^\star_{10}$ and a commutative diagram
$$\xymatrix@C=50pt@M=7pt@R=8pt
{
 \cN^\star_{10}\ar[d]_{\wp^*}&\partial \cN^\star_{10}\ar[d]_{\partial\wp^*}\ar@<0.5ex>@{_{(}->}[l]
 \ar@<0.6ex>[dr]^-{\pi}
 \ar@<-0.6ex>[dr]_-{\pi^\star}\\
\cA^\star_{10}& \partial\cA^\star_{10}\ar[d]_p\ar@<0.5ex>@{_{(}->}[l] \ar[d] &\left\{ {\begin{matrix}\hbox{connected double}\\\hbox{\'etale covers}\\\hbox{of plane sextics}\end{matrix}}\right\}=\cP\ar[dl]^-{\rm Prym}\\
&\cA_9
}
$$
where the map $ p\circ \partial\wp^*={\rm Prym}\circ \pi={\rm Prym}\circ \pi^\star$ sends a nodal threefold of type $\cX_{10}$ to the intermediate Jacobian of its normalization.

A dimension count shows that $  \partial\cN^\star_{10}$ is irreducible of dimension 21. By \cite{ver}, Corollary 4.10, the map ${\rm Prym}$ is 2-to-1 onto its 19-dimensional image. The irreducible variety $\cP$ therefore carries a birational involution $\sigma$ and $\pi^\star=\sigma\circ\pi$. By  Corollary \ref{reco}, the general fiber of $\pi$ is birationally the surface $S^{\rm odd}/\sigma$ so, with the notation of \S\ref{vs}, if $T$ is a general Verra solid, the fiber $(p\circ \partial\wp^*)^{-1}(J(T))$ is birationally the union of the two special surfaces $S^{\rm odd}_1/\sigma_1$ and $S_2^{\rm odd}/\sigma_2$.

We proved in \cite{dim}, Theorem 7.1, that   $\wp^* (\cN^\star_{10})$ has dimension 20. It follows  that $ \partial\wp^* (\partial\cN^\star_{10})$ has dimension 19, hence the  fibers of $\pi$ must be contracted by $ \partial \wp^*$ . This can also be seen by checking that the various nodal threefolds  corresponding to general points in this fiber differ by conic transformations (\S\ref{ct}) hence have same intermediate Jacobians (there is a birational isomorphism between them which is define at the nodes). 

From Lemma \ref{lemp}, we deduce that a line transform of  a nodal $X$ produces another nodal $X_\ell$ such that the   covers $\pi(X_\ell)$ and $\pi^\star(X)$ are isomorphic. Since their intermediate Jacobians are isomorphic (there is a birational isomorphism between them which is defined at the nodes), the   surfaces $S^{\rm odd}_1/\sigma_1$ and $S_2^{\rm odd}/\sigma_2$ are in the same fiber of $\wp^\star$.

 It follows that $p$ is birational on $\wp^* (\partial\cN^\star_{10})$:  the extension class  $e $ is canonically attached to  a general Verra solid. Moreover, since a general nonempty fiber of $ \partial\wp^*$ is the union of two distinct   irreducible surfaces (obtained by conic and line transformations), {\em the two smooth irreducible  surfaces   contained in a general nonempty fiber of $\wp^\star$} constructed in \cite{dim} {\em are disjoint}.\footnote{Unfortunately, we cannot deduce from our description  of a general nonempty fiber of $\partial\wp^*$ that a general nonempty fiber of $\wp^\star$ consists only of these two surfaces, since there are serious issues of properness here.}

  \section{Singularities of the theta divisor}\label{sings}
  
We keep the notation of \S\ref{vs}.  In \cite{ver}, Theorem 4.11, Verra proves that the singular locus of a theta divisor $\Xi$ of the intermediate Jacobian of a general Verra solid $T$   has (at least) three components, all $3$-dimensional:
  $$\Sing^{\rm ex}_{\pi_i}(\Xi)=\pi_i^*H_i+ \widetilde\Gamma_{6,i}+S_i^{\rm odd}$$ 
 for $i\in\{1,2\}$, 
  and another (union of) component(s) 
  $$\Xi_T\subset  \Sing^{\rm st}_{\pi_1}(\Xi)\cap \Sing^{\rm st}_{\pi_2}(\Xi)$$
   which Verra calls {\em distinguished.}\footnote{\label{oe} 
   Verra does not prove that $\Xi_T$ is irreducible (\cite{ver}, Remark 4.12). Recall that the singularities of a theta divisor $\Xi$ of the Prym variety of a double \'etale covering $\pi$ are of two kinds: the {\em stable  singularities} $\Sing^{\rm st}_\pi(\Xi)$, and the {\em exceptional singularities} $\Sing^{\rm ex}_\pi(\Xi)$ (see \cite{deb}, \S2, for the definitions).}

What follow    are some speculations on the singular locus of the intermediate Jacobian  of a general Fano threefold $X'$ of type $\cX_{10}$. We proved in \cite{dim} that a general fiber of the period map has two connected components $F(X')$ and $F^\star(X')$, which are smooth proper surfaces.

\begin{conj}\label{cs}
The singular locus of the theta divisor of the intermediate Jacobian of a general threefold $X'$ of type $\cX_{10}$ has dimension 4 and contains a unique   component of that dimension; this component is a translate of   $F(X')+ F^\star(X')$. 
\end{conj}

When $X'$ degenerates to a nodal $X$, with associated Verra threefold $T$,  the Fano surface $F(X')$ degenerates to the special surface $S_1^{\rm odd}$, and the surface $F^\star(X')$ to $S_2^{\rm odd}$. If $(J(T),\Xi)$ is the intermediate Jacobian of $T$, the singular locus of the theta divisor 
  degenerates to a subvariety of $\overline{J(X)}$ (see Remark \ref{compact}) which projects onto $\Sing (\Xi\cdot \Xi_e)$ (see \cite{mum}, (2.4), pp.~363--364). So the degenerate version of the conjecture is the following.
  
  \begin{conj}
Let $(J(T),\Xi)$ be the  intermediate Jacobian of a general Verra threefold $T$, with canonical extension class $e$. The singular locus of $\Xi\cdot \Xi_e$ has dimension 4 and has a unique component of that dimension. This   component   is a translate of  $S_1^{\rm odd}+S_2^{\rm odd}$.
\end{conj}

According to \S\ref{extex}  and (\ref{tra}),  there is   an embedding $ \widetilde\Gamma_{6,2}\subset S_1^{\rm odd}\subset \widetilde \Gamma_{6,1}^{(6)}$ such that 
for  all  $\widetilde D \in \widetilde \Gamma_{6,1}^{(6)}$ in this curve,
 $ \widetilde D +e $  is linearly equivalent to an effective divisor $\widetilde D'\in S_1^{\rm odd}$.
This implies
 $$\pi_1^*H_1 +e\equiv \sigma_1^*\widetilde D+\widetilde D+e\equiv \sigma_1^*\widetilde D + \widetilde D',$$
hence $h^0(\widetilde \Gamma_{6,1}, \pi_1^*H_1 + e)\ge 2$. It follows that
$   \Sing^{\rm ex}_{\pi_1}(\Xi ) 
$, hence also 
$    \Sing^{\rm ex}_{\pi_2} (\Xi)  $, is contained in $ \Xi_{e}
$. In particular, $   \Sing^{\rm ex}_{\pi_i}(\Xi ) $ is contained in the singular locus of $\Xi\cdot \Xi_e$, and is also contained in a translate of $S_1^{\rm odd}+S_2^{\rm odd}$ by (\ref{tra}).

\section{Appendix: explicit computations}\label{coor}

 \subsection{The fourfold $W$}\label{coow}
 It follows from \cite{pv}, Proposition 6.4, that there exists a basis $(e_1,\dots,e_5)$ for $V_5$ such that the pencil of skew-symmetric forms on $V_5$ that defines $W$ in $G(2,V_5)$ (see \S\ref{fw}) is spanned by $e_1^*\wedge e_4^*+e_2^*\wedge e_5^*$ and $e_1^*\wedge e_5^*+e_3^*\wedge e_4^*$. In other words, in coordinates for the basis 
 $$\cB=(e_{12},e_{13},e_{14},e_{15},e_{23},e_{24},e_{25},e_{34},e_{35},e_{45})$$
 for $\wedge^2V_5$, where $e_{ij} =e_i\wedge e_j$, the linear space $V_8$ that cuts out $W$ has equations
 $$x_{14}+x_{25}=x_{15}+x_{34}=0.$$
 The unique common maximal isotropic subspace for all forms in the pencil is
 $$U_3=\langle e_1,e_2,e_3\rangle.$$
 It contains the 
 smooth conic 
 $$c_U=(x_1^2+x_2x_3=x_4=x_5=0),$$ which parametrizes the kernels of the forms in the pencil. The unique $\beta$-plane (see \cite{dim}, \S 3.3) contained in $W$ is
 $$\Pi=G(2,U_3)=\langle  e_{12},e_{13},e_{23}\rangle$$
 and the orbits for the action of $\Aut(W)$ are (see \cite{dim}, \S 3.5)
 \begin{itemize}
\item $O_1=c_U^\vee=  (x_{23}^2+4x_{12}x_{13} =0) \subset\Pi$,
\item  $O_2=\Pi\moins c_U^\vee$,
\item  $O_3=(W\cap (x_{45}=0))\moins\Pi$,
\item   $O_4=W\moins O_3$. 
\end{itemize}

 \subsection{The fourfold $W_O$}\label{coowo}
The point $O=[e_{45}]$ is in the dense orbit  $O_4$, which can be parametrized by
$$ O_4=\{ (-ux-v^2,u^2-vy,-v,u,uv+xy,x,v,-u,y,1)\mid (u,v,x,y)\in \C^4\}
$$
in the basis $\cB$, and  $\T_{W,O}\subset \P_7$ has equations
$$x_{12}=x_{13}=x_{23}=x_{14}+x_{25}=x_{15}+x_{34}=0.$$
In particular,  $\T_{W,O}\cap\Pi=\vide$.

   The quadrics that contain $W$ and are singular at $O$ correspond, via the isomorphism (\ref{iso}), to one-dimensional subspaces $V_1\subset V_O=\langle e_4,e_5\rangle$ (see \S\ref{iden}). It follows that the projection $W_O\subset \P^6_O$  is defined by  
\begin{equation}\label{pencil}
x_{12}x_{34}-x_{13}x_{24}+x_{14}x_{23}=
   x_{12}x_{35}+x_{13}x_{14}-x_{34}x_{23}=0,
   \end{equation}
   where  $(x_{12},x_{13},x_{14},x_{23} ,x_{24},x_{34} ,x_{35})
   $  are   coordinates
   for $V_8/ \langle e_{45}\rangle$.
   It contains the $3$-plane
   $$\P^3_W=p_O(\T_{W,O})= (x_{12}=x_{13}=x_{23}=0) $$
and its singular locus is the twisted cubic $C_O\subset \P^3_W$ defined by
\begin{equation}\label{twcu} \rank\begin{pmatrix} x_{14}& x_{34}& -x_{35}\\ x_{24}&x_{14}&  x_{34}
\end{pmatrix}\le 1.
\end{equation}
Let $ \widetilde \P^6_O\to \P^6_O$ the blow-up of $\P^3_W$, where $ \widetilde \P^6_O\subset \P^6_O\times \P^2_W$ is defined by the relation
\begin{equation}\label{sys1} \rank\begin{pmatrix} a_{12}& a_{13}& a_{23}\\ x_{12}& x_{13}& x_{23}
\end{pmatrix}\le 1,
\end{equation}
where $a_{12}, a_{13},a_{23}$ are homogeneous coordinates on $\P^2_W$.
The strict transform $\widetilde \P^3_W\subset \widetilde \P^6_O$ of $\P^3_W$ is isomorphic to the blow-up of $C_O$, and the strict transform $\widetilde W_O\subset \widetilde \P^6_O$ of $W_O$  is   defined  by the  equations 
\begin{equation}\label{sys}
a_{12}x_{34}-a_{13}x_{24}+ a_{23}x_{14}=
  a_{12}x_{35}+a_{13}x_{14}- a_{23}x_{34}=0 .
\end{equation}
 It follows that $ \widetilde W_O\to W_O$ is a $\P^1$-bundle over $C_O$ and an isomorphism over $W_O\moins C_O$. Furthermore, the projection  $\widetilde W_O\to \P^2_W$ 
  is a $\P^2$-bundle, 
 $\widetilde W_O$ is smooth, and its restriction to $\widetilde \P^3_W $ 
  is a $\P^1$-bundle.

\subsection{The $\P^2$-bundles $\P(\cM_O)\to \P^2_O$  and  $\widetilde W_O\to \P^2_W$ are   isomorphic}\label{wtw}
If $V_4\subset V_5$ is a hyperplane that contains $V_O$ and is defined by the equation $b_1x_1+b_2x_2+b_3x_3=0$, the vector space $M_{V_4}= \wedge^2 V_4 \cap V_8$ defined in \S\ref{iden} has equations
$$b_1x_{14}+b_2x_{24}+b_3x_{34}=-b_1x_{34}-b_2x_{14}+b_3x_{35}=0$$
and 
$$b_2x_{12}+b_3x_{13}=b_1x_{12}-b_3x_{23}=b_1x_{13}+b_2x_{23}=0$$
 in $V_8$. It is therefore equal to the fiber of $\widetilde W_O\to \P^2_W$ at the point $a= (b_3,-b_2,b_1)$ (see (\ref{sys1}) and (\ref{sys})). This isomorphism $\P^2_O\isom \P^2_W  $ induces an isomorphism $\P^2_O\times\P^6_O\isom \P^2_W  \times\P^6_O$, hence an isomorphism 
between the $\P^2$-bundles $\P(\cM_O)\to \P^2_O$  and  $\widetilde W_O\to \P^2_W$.

\end{document}